\documentclass[12pt]{article}
\usepackage{amsmath,amsthm}
\usepackage{amssymb}

\renewcommand{\leq}{\leqslant}

\renewcommand{\le}{\leqslant}

\def\cref#1{Corollary~$\ref{#1}$}

\input epsf

\usepackage{
amsmath,  amsthm, amssymb, amsthm}

\usepackage{graphics}
\usepackage{epsfig}

\usepackage{amsmath}
\usepackage{amsfonts}
\usepackage{amssymb}

\newtheorem{prop}{Proposition}[section]

\newcommand{\bgeqn}{\begin{eqnarray}}
\newcommand{\edeqn}{\end{eqnarray}}
\newcommand{\bgeq}{\begin{eqnarray*}}
\newcommand{\edeq}{\end{eqnarray*}}
\newcommand{\bec}{\begin{center}}
\newcommand{\enc}{\end{center}}

\newcommand{\be}{\begin{equation}}
\newcommand{\ee}{\end{equation}}

\usepackage{amsfonts}
\usepackage{epsfig}

\usepackage{amsmath}
\usepackage{amsthm}
\usepackage{amssymb}
\usepackage{epsfig}
\usepackage{color}

\theoremstyle{plain}

\begin{document}

\title{\bf A Simple Approach to Constructing Quasi-Sudoku-based Sliced Space-Filling Designs}

\author{
Diane Donovan$^1$, Benjamin Haaland$^{2,3}$, David J. Nott$^4$  \\
$^1$ University of Queensland, Brisbane, Australia\\
$^2$Georgia Institute of Technology, USA\\
$^3$Duke-NUS Graduate Medical School, Singapore\\
$^3$National University of Singapore, Singapore}

\date{\today}

\maketitle

\begin{abstract}
\noindent Sliced Sudoku-based space-filling designs and, more generally, quasi-sliced orthogonal array-based space-filling designs are useful experimental designs in several contexts, including computer experiments with categorical in addition to quantitative inputs and cross-validation. Here, we provide a straightforward construction of doubly orthogonal quasi-Sudoku Latin squares which can be used to generate sliced space-filling designs which achieve uniformity  in one and two-dimensional projections for both the full design and each slice. A construction of quasi-sliced orthogonal arrays based on these constructed doubly orthogonal quasi-Sudoku Latin squares is also provided and can, in turn, be used to generate sliced space-filling designs which achieve uniformity  in one and two-dimensional projections for the full design and and uniformity in two-dimensional projections for each slice. These constructions are very practical to implement and yield a spectrum of design sizes and numbers of factors not currently broadly available.
\end{abstract}

\noindent
KEY WORDS:
Computer experiments, space-filling designs, Sudoku, sliced experimental designs.

\section{Introduction}

In the popular game Sudoku, players are presented with a
nine-by-nine array, divided into nine three-by-three subsubarrays,
and partially filled with the numbers $1$ through $9$. The goal is
to fill the nine-by-nine array with the numbers $1$ through $9$ so
that each row, column, and three-by-three subarray contains no
repeated numbers. See a starting and completed Sudoku square in
Figure \ref{SudokuGrid} \cite{usatoday}.

\begin{figure}[!h]\label{SudokuGrid}
\centering
  \begin{tabular}{|ccc|ccc|ccc|}
\hline
&&&\bf 7&\bf 3&\bf 5&\bf 8&\bf 2&\\
\bf 8&&&&\bf 1&&&&\bf 5\\
&&&&\bf 2&\bf 6&&\bf 1&\\
\hline
&\bf 8&\bf 1&&\bf 7&&&&\bf 9\\
&\bf 6&&&\bf 9&&&\bf 5&\\
\bf 9&&&&\bf 5&&\bf 3&\bf 8&\\
\hline
&\bf 3&&\bf 1&\bf 4&&&&\\
\bf 4&&&&\bf 6&&&&\bf 3\\
&\bf 9&\bf 2&\bf 3&\bf 8&\bf 7&&&\\
\hline
  \end{tabular}\quad\quad\quad
\begin{tabular}{|ccc|ccc|ccc|}
\hline
1&4&9&\bf 7&\bf 3&\bf 5&\bf 8&\bf 2&6\\
\bf 8&2&6&9&\bf 1&4&7&3&\bf 5\\
3&5&7&8&\bf 2&\bf 6&9&\bf 1&4\\
\hline
5&\bf 8&\bf 1&2&\bf 7&3&4&6&\bf 9\\
2&\bf 6&3&4&\bf 9&8&1&\bf 5&7\\
\bf 9&7&4&6&\bf 5&1&\bf 3&\bf 8&2\\
\hline
7&\bf 3&5&\bf 1&\bf 4&2&6&9&8\\
\bf 4&1&8&5&\bf 6&9&2&7&\bf 3\\
6&\bf 9&\bf 2&\bf 3&\bf 8&\bf 7&5&4&1\\
\hline
  \end{tabular}
\caption{A starting and completed Sudoku square \cite{usatoday}.}
\end{figure}

Sets of (completed) Sudoku squares, as well as generalizations
thereof,  can be used to construct \emph{sliced space-filling
designs} achieving maximal uniformity in both one and
two-dimensional projections for both the complete design and each
subdesign, or slice. Sliced space-filling designs are experimental
designs which can be partitioned into groups of subdesigns so that
both the full design and each subdesign achieve some type of
uniformity. These types of designs are broadly useful for
collecting data from computer experiments, large and
time-consuming mathematical codes used to model real world systems
such as the climate or a component in an engineering design
problem. Sliced space-filling designs are particularly useful for
computer experiments with qualitative and quantitative inputs
\cite{qianWu}, multiple levels of accuracy \cite{haalandQian}, and
cross-validation problems in the context of
computer experiments \cite{zhangQian}. Sudoku-based sliced space-filling designs
were introduced in \cite{XuHaalandQian} and a construction was
given using doubly orthogonal Sudoku squares, whose complete
arrays are orthogonal and whose subarrays are orthogonal after a
projection. In \cite{XuHaalandQian}, doubly orthogonal Sudoku
squares were constructed using the techniques developed in
\cite{PedersenVis} along with a subfield projection.

Here, we give a straightforward construction of doubly
orthogonal \emph{quasi}-Sudoku  squares relying on the existence
of sets of orthogonal Latin squares, which are relatively
well-described and available for a broad range of sizes
\cite{colbourn,Denes,Raghavarao}.
Further, the presented construction of sliced Sudoku-based
space-filling designs is available for a spectrum of sample
sizes and number of factors combinations.
The
remainder of this article is organized as follows. Section
\ref{notationDef} provides notation and definitions which will be
used throughout. Section \ref{construction} provides a
construction for sets of pairwise doubly orthogonal quasi-Sudoku Latin
squares which is based on sets of orthogonal Latin squares.
Section \ref{example} illustrates the presented techniques with an
example. Section \ref{sec:oa} notes the connection with
quasi-sliced asymmetric orthogonal arrays. Finally, Section
\ref{construction2} reviews the construction of Sudoku-based
sliced space-filling designs from \cite{XuHaalandQian}.

\section{Notation and Definitions}\label{notationDef}

Let $[n]=\{0,1,\dots,n-1\}$ and $A=[A(i,j)]$ be a {\em Latin
square} of order $n$; that is,  an $n\times n$ array in which
each of the entries in a set $N$ (usually $[n]$) of size $n$
occurs once in every row and once in every column. Two Latin
squares $A=[A(i,j)]$ and $B=[B(i,j)]$, of the same order, are said
to be {\em orthogonal} if, when we superimpose one on top of the
other, the arrays contain each of the $n^2$ ordered pairs
$(x,y)$, $x,y\in N$ exactly once.
It is useful to note that
$A$ and $B$ are orthogonal if and only if for all
$p,q,s,t\in N$
\begin{eqnarray*}
A(p,q)=A(s,t)\Longrightarrow
B(p,q)\neq B(s,t).
\end{eqnarray*}
Let  $A=[A(i,j)]$ and $B=[B(i,j)]$ be  two 
orthogonal Latin squares of order $n$. For $n=s^2$, let $\Pi$
denote a projection from $N$ to $[s]$, $\Pi:N\rightarrow [s]$,
and let ${\cal O}=\{(\Pi(A(i,j)),\Pi(B(i,j)))\mid i,j\in N\}$.
We may think of ${\cal O}$ as an $n\times n$ array obtained by
superimposing $\Pi(A)$ and $\Pi(B)$. The Latin squares $A$ and $B$
are said to be {\em doubly orthogonal} if 
there exists a projection $\Pi$ such that
${\cal O}$ can be
partitioned into $s\times s$ subarrays with the cells of each
subarray containing the $s^2$ ordered pairs $(x,y)$, $0\leq x,y\leq
s-1$. For $n=rs$, let $\Pi_r$ denote a projection from $N$ to
$[r]$, $\Pi_r:N\rightarrow [r]$, let $\Pi_s$ denote a projection
from $N$ to $[s]$, $\Pi_s:N\rightarrow [s]$, and ${\cal
O}=\{(\Pi_r(A(i,j)),\Pi_s(B(i,j)))\mid i,j\in N\}$. The Latin
squares $A$ and $B$ are said to be {\em  doubly orthogonal} if
there exist projections $\Pi_r$ and $\Pi_s$ such that
${\cal O}$ can be partitioned into $r\times s$ subarrays with
the cells of each subarray containing the $rs$ ordered pairs $(x,y)$
where $0\leq x\leq r-1$ and $0\leq y\leq s-1$.

An $m^2\times m^2$ array is said to be a {\em Sudoku Latin
square},  on the set $X$ of size $m^2$, if it is a Latin square
and we can label the rows  by $(p,s)$, $0\leq p,s\leq m-1$ and
columns by $(q,t)$, $0\leq q,t\leq m-1$ such that  for each $p$
and $q$ the subarray defined by the cells
\begin{eqnarray*}
((p,s),(q,t)),&0\leq s,t\leq m-1,
\end{eqnarray*}
contains each entry of $X$ precisely once.
 A $mn\times mn$ array is said be a {\em quasi-Sudoku Latin square},
  on the set $X$ of order $mn$, if it is a Latin square and we can
  label the rows  by $(p,s)$, $0\leq p\leq m-1$, $0\leq s\leq n-1$
   and columns $(q,t)$, $0\leq q\leq n-1$, $0\leq t\leq m-1$ such that  for each $p$ and $q$ the subarray defined by the cells
\begin{eqnarray*}
((p,s),(q,t)),&0\leq s\leq n-1,&0\leq t\leq m-1,
\end{eqnarray*}
contains each entry of $X$ precisely once. Two Sudoku Latin
squares or quasi-Sudoku Latin squares are orthogonal (doubly
orthogonal) if they are orthogonal (doubly orthogonal) Latin
squares.

\section{Construction of Doubly Orthogonal Quasi-Sudoku Latin Squares}\label{construction}

Here, we construct quasi-Sudoku Latin squares
which
are doubly orthogonal using sets of (pairwise) orthogonal Latin squares.
This is done by first using a direct
product construction to construct orthogonal quasi-Sudoku Latin
squares and then showing that they are doubly orthogonal.

For the remainder of this section we will take $A_1$ and $A_2$ to be
two Latin squares of order $m$ and $n$, respectively.  We can
construct a new Latin square of order $mn$ by taking the {\em
direct product}, $A_1\otimes A_2$, of $A_1$ with $A_2$,
where $(A_1(p,q),A_2(s,t))$ is the element in row $np+s$ and column $nq+t$ of $A_1\otimes A_2$, $0\le p,q\le m-1$ and $0\le s,t\le n-1$.
Label the rows of $A_1\otimes A_2$ as $(p,s)$ and the columns of
$A_1\otimes A_2$ as $(q,t)$.
For fixed $p$ and $q$, the subarray defined by the set of
cells $\{((p,s),(q,t))\mid 0\leq s,t\leq n-1\}$ is then isomorphic to
$A_2$ .

The
next proposition attests that orthogonality is maintained under the direct product. For a proof, see D\'{e}nes and Keedwell \cite{Denes} page
427.

\begin{prop}\label{orth prods}
If $A_1$ and $B_1$ are orthogonal Latin squares of order $m$ and
$A_2$ and $B_2$ are orthogonal Latin squares of order $n$, then
$A_1\otimes A_2$ and $B_1\otimes B_2$ are orthogonal Latin
squares of order $mn$.
\end{prop}

An example of orthogonal Latin squares $A_1\otimes A_2$ and
$B_1\otimes B_2$, constructed using the orthogonal Latin squares shown in Figure \ref{MOLS34}, is
described in Section \ref{example} and
shown in Figure \ref{ex:orthog}.

\begin{prop}\label{double sudoku}
If $A_1$ and $B_1$ are orthogonal Latin squares of order $m$ and
$A_2$ and $B_2$ are orthogonal Latin squares of order $n$, then $A_1\otimes A_2$ and $B_1\otimes B_2$ are doubly orthogonal quasi-Sudoku Latin squares of order
$mn$.
\end{prop}

\begin{proof}
We begin by verifying that there exists an arrangement of the rows of both arrays
$A_1\otimes A_2$ and  $B_1\otimes B_2$ which verifies that they
are quasi-Sudoku latin squares.

For fixed $q$, the entries $A_1(p,q)$ and  $B_1(p,q)$  take the
values in column $q$ of $A_1$ and $B_1$ respectively.  That is,
for fixed $q$, $\{A_1(p,q)\mid0\leq p\leq m-1\}=\{B_1(p,q)\mid
0\leq p\leq m-1\}=[m]$. Likewise, for fixed $s$, the entries
$A_2(s,t)$ and $B_2(s,t)$ take values in row $s$ of $A_2$ and
$B_2$, respectively. That is, for fixed $s$, $\{A_2(s,t)\mid 0\leq
t\leq n-1\}=\{B_2(s,t)\mid 0\leq t\leq n-1\}=[n]$. Thus for fixed $q$ and $s$
\begin{eqnarray*}
[m]\times [n]&=&\{(A_1(p,q),A_2(s,t))\mid 0\leq p\leq m-1,0\leq
t\leq n-1\}\\&=& \{(B_1(p,q),B_2(s,t))\mid 0\leq p\leq m-1,0\leq
t\leq n-1\}.
\end{eqnarray*}
Hence we will assume that the rows of $A_1\otimes A_2$ and
$B_1\otimes B_2$ have been reordered by moving row $np+s$ to position $ms+p$, where $0\leq p\leq
m-1$ and $0\leq s\leq n-1$. The order of the columns, however,
remains as $qn+t$, where $0\leq q\leq m-1$ and $0\leq t\leq n-1$.
Now  for fixed $s$ and $q$ the subarrays of $A_1\otimes A_2$ and
$B_1\otimes B_2$ defined by the intersection of rows $ms+p$ with
columns $qn+t$, where $0\leq p\leq m-1$ and $0\leq t\leq n-1$
contains each of the $mn$ entries precisely once. Note that the
reordering of rows has been consistently applied to both
$A_1\otimes A_2$ and $B_1\otimes B_2$, therefore these Latin
squares are still orthogonal.

We select two onto functions $\Pi_m:[m]\times [n]\rightarrow [m]$ and $\Pi_n:[m]\times [n]\rightarrow [n]$ such that 
 if 
$\Pi_m(A_1(p,q),A_2(s,t))=\Pi_m(A_1(p,q),A_2(s,t'))$ then $t\neq t'$ and if
$\Pi_n(B_1(p,q),B_2(s,t))=\Pi_n(B_1(p,q),B_2(s,t'))$ then $t=t'$. Such functions are not hard to find for instance we could take
\begin{eqnarray*}
\Pi_m(A_1(p,q),A_2(s,t))&=&A_1(p,q)\mbox{ and }\\
\Pi_n(B_1(p,q),B_2(s,t))&=&B_2(s,t),
\end{eqnarray*}
or if $m$ and $n$ are coprime, with $m>n$, take
\begin{eqnarray*}
\Pi_m(A_1(p,q),A_2(s,t))&=&(n\times A_1(p,q)+A_2(s,t))\mbox{ mod }m\mbox{ and }\\
\Pi_n(B_1(p,q),B_2(s,t))&=&(n\times B_1(p,q)+B_2(s,t))\mbox{ mod
}n.
\end{eqnarray*}
In this latter case since $m$ and $n$ are coprime and $m>n$, for
fixed $q$, $\{n\times A_1(p,q)\mbox{ mod }m\mid 0\leq p\leq
m-1\}=[m]$ and since $A_2(s,t)<m$,  $\{n\times
A_1(p,q)+A_2(s,t)\mbox{ mod }m\mid 0\leq p\leq m-1\}=[m]$, so for
any fixed $q$, $s$ and $t$, $\{\Pi_m((A_1(p,q),A_2(s,t)))\mid 0\leq
p\leq m-1\} =[m]$. Further,
$\{(\Pi_m((A_1(p,q),A_2(s,t))),(\Pi_n((B_1(p,q),B_2(s,t)))\mid
0\leq p\leq m-1,0\leq t\leq n\} =[m]\times [n] $. This verifies
that $A_1\otimes A_2$ and $B_1\otimes B_2$ are doubly orthogonal
quasi-Sudoku squares.

\end{proof}

It should be noted that the different projections given in the above proof produce non-isomorphic squares. This will be illustrated in the example given below. 
For $m=n$, the extension of these pairwise properties to more than two orthogonal direct product designs is immediate, if the component designs are available.
Further, it should be noted that when we write $A_1\otimes A_2$ and
$B_1\otimes B_2$ are orthogonal quasi-Sudoku Latin squares we are assuming that the rows of the direct products have been rearranged to the required format for Sudoku Latin squares.

\section{An Example}\label{example}


As an illustration, we will construct doubly orthogonal quasi-Sudoku Latin squares of
order 12. We begin with two orthogonal Latin squares of order $n=3$
and two of the three orthogonal Latin squares of order $m=4$, as shown in Figure \ref{MOLS34}. Note
it does not matter which two orthogonal Latin squares of order $4$
we choose, so we arbitrarily select the first and the last. In general,
if the underlying Latin squares are non-isomorphic then it is
possible to construct sets of non-isomorphic doubly orthogonal
quasi-Sudoku Latin squares.
\begin{figure}[!h]
\begin{center}
\begin{tabular}{ccccc}
$A_1$&&$B_1$&$A_2$&$B_2$\\
\begin{tabular}{|cccc|}
\hline
0&1&2&3 \\
1&0&3&2 \\
2&3&0&1 \\
3&2&1&0\\
\hline
\end{tabular}&
\begin{tabular}{|cccc|}
\hline
0&1&2&3 \\
3&2&1&0\\
1&0&3&2 \\
2&3&0&1 \\
\hline
\end{tabular}&
\begin{tabular}{|cccc|}
\hline
0&1&2&3\\
2&3&0&1\\
3&2&1&0\\
1&0&3&2\\
\hline
\end{tabular}
&
\begin{tabular}{|ccc|}
\hline
0&1&2   \\
1&2&0\\
2&0&1
\\
\hline
\end{tabular}&
\begin{tabular}{|ccc|}
\hline 0&1&2
\\
2&0&1
\\
    1&2&0
\\
\hline
\end{tabular}
\end{tabular}
\end{center}
\caption{The three (pairwise) orthogonal Latin squares of order 4 and the two orthogonal Latin squares of order 3.}\label{MOLS34}
\end{figure}
In Figure \ref{ex:orthog} we construct $A_1\otimes A_2$ and
$B_1\otimes B_2$ using the direct product construction. To
facilitate understanding, rows and columns have been labelled and
$(x,y)$ has been replaced by $xy$.
  In Figure \ref{ex:sudoku} the rows have been reordered
as first $(p,0)$ for $0\leq p\leq 3$, then $(p,1)$ for $0\leq p\leq 3$, then $(p,2)$ for $0\leq p\leq 3$
to emphasise the fact that
  these squares are quasi-Sudoku Latin squares.

In Figure
  \ref{ex:proj} we project down each subsquare in Figure \ref{ex:sudoku}  by applying the
mappings $\Pi_4: [4]\times
[3]\rightarrow [4]$ and $\Pi_3:[4]\times [3]\rightarrow [3]$,
given by
\begin{eqnarray*}
\Pi_4((A_1(p,q),A_2(s,t)))&=&(3\times A_1(p,q)+A_2(s,t))\mbox{ mod }4,\mbox{ and }\\
\Pi_3((B_1(p,q),B_2(s,t)))&=&(3\times B_1(p,q)+B_2(s,t))\mbox{ mod
}3,
\end{eqnarray*}
to the entries in the quasi-Sudoku squares obtained by reordering the rows of the Latin squares $A_1\otimes A_2$ and
$B_1\otimes B_2$
and label these as
$\Pi_4(A_1\otimes A_2)$ and $\Pi_3(B_1\otimes B_2)$.
Then, in Figure
  \ref{ex:superimposition} we superimpose the projections $\Pi_4(A_1\otimes
A_2)$ and $\Pi_3(B_1\otimes B_2)$ to
  verify that indeed  $A_1\otimes A_2$ and
  $B_1\otimes B_2$ are doubly orthogonal quasi-Sudoku Latin
  squares of order 12, as each $4\times 3$ subsquare has each of
the ordered pairs $(x,y)$, $0\leq x\leq 3$ and $0\leq y\leq 2$.

\begin{figure}[!h]
\begin{center}
\begin{tabular}{c}
$A_1\otimes A_2$\\
\begin{tabular}{c|ccc|ccc|ccc|ccc|}
\multicolumn{1}{c}{}&\multicolumn{1}{c}{00}&\multicolumn{1}{c}{01}&\multicolumn{1}{c}{02}&\multicolumn{1}{c}{10}&
\multicolumn{1}{c}{11}&\multicolumn{1}{c}{12}&\multicolumn{1}{c}{20}&\multicolumn{1}{c}{21}&\multicolumn{1}{c}{22}&\multicolumn{1}{c}{30}&\multicolumn{1}{c}{31}&\multicolumn{1}{c}{32}\\
\cline{2-13}
\multicolumn{1}{r|}{00}&00&01&02&10&11&12&20&21&22&30&31&32\\
\multicolumn{1}{r|}{01}&01&02&00&11&12&10&21&22&20&31&32&30\\
\multicolumn{1}{r|}{02}&02&00&01&12&10&11&22&20&21&32&30&31\\
\cline{2-13}
\multicolumn{1}{r|}{10}&10&11&12&00&01&02&30&31&32&20&21&22\\
\multicolumn{1}{r|}{11}&11&12&10&01&02&00&31&32&30&21&22&20\\
\multicolumn{1}{r|}{12}&12&10&11&02&00&01&32&30&31&22&20&21\\
\cline{2-13}
\multicolumn{1}{r|}{20}&20&21&22&30&31&32&00&01&02&10&11&12\\
\multicolumn{1}{r|}{21}&21&22&20&31&32&30&01&02&00&11&12&10\\
\multicolumn{1}{r|}{22}&22&20&21&32&30&31&02&00&01&12&10&11\\
\cline{2-13}
\multicolumn{1}{r|}{30}&30&31&32&20&21&22&10&11&12&00&01&02\\
\multicolumn{1}{r|}{31}&31&32&30&21&22&20&11&12&10&01&02&00\\
\multicolumn{1}{r|}{32}&32&30&31&22&20&21&12&10&11&02&00&01\\
\cline{2-13}
\end{tabular}\\
$ $\\
$B_1\otimes B_2$\\
\begin{tabular}{c|ccc|ccc|ccc|ccc|}
\multicolumn{1}{c}{}&\multicolumn{1}{c}{00}&\multicolumn{1}{c}{01}&\multicolumn{1}{c}{02}&\multicolumn{1}{c}{10}&
\multicolumn{1}{c}{11}&\multicolumn{1}{c}{12}&\multicolumn{1}{c}{20}&\multicolumn{1}{c}{21}&\multicolumn{1}{c}{22}&\multicolumn{1}{c}{30}&\multicolumn{1}{c}{31}&\multicolumn{1}{c}{32}\\
\cline{2-13}
\multicolumn{1}{r|}{00}&00&01&02&10&11&12&20&21&22&30&31&32\\
\multicolumn{1}{r|}{01}&02&00&01&12&10&11&22&20&21&32&30&31\\
\multicolumn{1}{r|}{02}&01&02&00&11&12&10&21&22&20&31&32&30\\
\cline{2-13}
\multicolumn{1}{r|}{10}&20&21&22&30&31&32&00&01&02&10&11&12\\
\multicolumn{1}{r|}{11}&22&20&21&32&30&31&02&00&01&12&10&11\\
\multicolumn{1}{r|}{12}&21&22&20&31&32&30&01&02&00&11&12&10\\
\cline{2-13}
\multicolumn{1}{r|}{20}&30&31&32&20&21&22&10&11&12&00&01&02\\
\multicolumn{1}{r|}{21}&32&30&31&22&20&21&12&10&11&02&00&01\\
\multicolumn{1}{r|}{22}&31&32&30&21&22&20&11&12&10&01&02&00\\
\cline{2-13}
\multicolumn{1}{r|}{30}&10&11&12&00&01&02&30&31&32&20&21&22\\
\multicolumn{1}{r|}{31}&12&10&11&02&00&01&32&30&31&22&20&21\\
\multicolumn{1}{r|}{32}&11&12&10&01&02&00&31&32&30&21&22&20\\
\cline{2-13}
\end{tabular}
\end{tabular}
\end{center}
\caption{A pair of orthogonal Latin squares of order $12$}\label{ex:orthog}
\end{figure}

\begin{figure}[h]
\begin{center}
\begin{tabular}{c}
Reordered $A_1\otimes A_2$\\
\begin{tabular}{c|ccc|ccc|ccc|ccc|}
\multicolumn{1}{c}{}&\multicolumn{1}{c}{00}&\multicolumn{1}{c}{01}&\multicolumn{1}{c}{02}&\multicolumn{1}{c}{10}&
\multicolumn{1}{c}{11}&\multicolumn{1}{c}{12}&\multicolumn{1}{c}{20}&\multicolumn{1}{c}{21}&\multicolumn{1}{c}{22}&\multicolumn{1}{c}{30}&\multicolumn{1}{c}{31}&\multicolumn{1}{c}{32}\\
\cline{2-13}
\multicolumn{1}{r|}{00}&00&01&02&10&11&12&20&21&22&30&31&32\\
\multicolumn{1}{r|}{10}&10&11&12&00&01&02&30&31&32&20&21&22\\
\multicolumn{1}{r|}{20}&20&21&22&30&31&32&00&01&02&10&11&12\\
\multicolumn{1}{r|}{30}&30&31&32&20&21&22&10&11&12&00&01&02\\
\cline{2-13}
\multicolumn{1}{r|}{01}&01&02&00&11&12&10&21&22&20&31&32&30\\
\multicolumn{1}{r|}{11}&11&12&10&01&02&00&31&32&30&21&22&20\\
\multicolumn{1}{r|}{21}&21&22&20&31&32&30&01&02&00&11&12&10\\
\multicolumn{1}{r|}{31}&31&32&30&21&22&20&11&12&10&01&02&00\\
\cline{2-13}
\multicolumn{1}{r|}{02}&02&00&01&12&10&11&22&20&21&32&30&31\\
\multicolumn{1}{r|}{12}&12&10&11&02&00&01&32&30&31&22&20&21\\
\multicolumn{1}{r|}{22}&22&20&21&32&30&31&02&00&01&12&10&11\\
\multicolumn{1}{r|}{32}&32&30&31&22&20&21&12&10&11&02&00&01\\
\cline{2-13}
\end{tabular}\\
$ $\\
Reordered $B_1\otimes B_2$\\
\begin{tabular}{c|ccc|ccc|ccc|ccc|}
\multicolumn{1}{c}{}&\multicolumn{1}{c}{00}&\multicolumn{1}{c}{01}&\multicolumn{1}{c}{02}&\multicolumn{1}{c}{10}&
\multicolumn{1}{c}{11}&\multicolumn{1}{c}{12}&\multicolumn{1}{c}{20}&\multicolumn{1}{c}{21}&\multicolumn{1}{c}{22}&\multicolumn{1}{c}{30}&\multicolumn{1}{c}{31}&\multicolumn{1}{c}{32}\\
\cline{2-13}
\multicolumn{1}{r|}{00}&00&01&02&10&11&12&20&21&22&30&31&32\\
\multicolumn{1}{r|}{10}&20&21&22&30&31&32&00&01&02&10&11&12\\
\multicolumn{1}{r|}{20}&30&31&32&20&21&22&10&11&12&00&01&02\\
\multicolumn{1}{r|}{30}&10&11&12&00&01&02&30&31&32&20&21&22\\
\cline{2-13}
\multicolumn{1}{r|}{01}&02&00&01&12&10&11&22&20&21&32&30&31\\
\multicolumn{1}{r|}{11}&22&20&21&32&30&31&02&00&01&12&10&11\\
\multicolumn{1}{r|}{21}&32&30&31&22&20&21&12&10&11&02&00&01\\
\multicolumn{1}{r|}{31}&12&10&11&02&00&01&32&30&31&22&20&21\\
\cline{2-13}
\multicolumn{1}{r|}{02}&01&02&00&11&12&10&21&22&20&31&32&30\\
\multicolumn{1}{r|}{12}&21&22&20&31&32&30&01&02&00&11&12&10\\
\multicolumn{1}{r|}{22}&31&32&30&21&22&20&11&12&10&01&02&00\\
\multicolumn{1}{r|}{32}&11&12&10&01&02&00&31&32&30&21&22&20\\
\cline{2-13}
\end{tabular}
\end{tabular}
\end{center}
\caption{A pair of orthogonal quasi-Sudoku Latin squares of order
$12.$}\label{ex:sudoku}
\end{figure}

\begin{figure}[!h]
\begin{center}
\begin{tabular}{c}
$\Pi_4(A_1\otimes A_2)$\\
\begin{tabular}{c|ccc|ccc|ccc|ccc|}
\multicolumn{1}{c}{}&\multicolumn{1}{c}{00}&\multicolumn{1}{c}{01}&\multicolumn{1}{c}{02}&\multicolumn{1}{c}{10}&
\multicolumn{1}{c}{11}&\multicolumn{1}{c}{12}&\multicolumn{1}{c}{20}&\multicolumn{1}{c}{21}&\multicolumn{1}{c}{22}&\multicolumn{1}{c}{30}&\multicolumn{1}{c}{31}&\multicolumn{1}{c}{32}\\
\cline{2-13}
\multicolumn{1}{r|}{00}&0&1&2&3&0&1&2&3&0&1&2&3\\
\multicolumn{1}{r|}{10}&3&0&1&0&1&2&1&2&3&2&3&0\\
\multicolumn{1}{r|}{20}&2&3&0&1&2&3&0&1&2&3&0&1\\
\multicolumn{1}{r|}{30}&1&2&3&2&3&0&3&0&1&0&1&2\\
\cline{2-13}
\multicolumn{1}{r|}{01}&1&2&0&0&1&3&3&0&2&2&3&1\\
\multicolumn{1}{r|}{11}&0&1&3&1&2&0&2&3&1&3&0&2\\
\multicolumn{1}{r|}{21}&3&0&2&2&3&1&1&2&0&0&1&3\\
\multicolumn{1}{r|}{31}&2&3&1&3&0&2&0&1&3&1&2&0\\
\cline{2-13}
\multicolumn{1}{r|}{02}&2&0&1&1&3&0&0&2&3&3&1&2\\
\multicolumn{1}{r|}{12}&1&3&0&2&0&1&3&1&2&0&2&3\\
\multicolumn{1}{r|}{22}&0&2&3&3&1&2&2&0&1&1&3&0\\
\multicolumn{1}{r|}{32}&3&1&2&0&2&3&1&3&0&2&0&1\\
\cline{2-13}
\end{tabular}\\
$ $\\
$\Pi_3(B_1\otimes B_2)$\\
\begin{tabular}{c|ccc|ccc|ccc|ccc|}
\multicolumn{1}{c}{}&\multicolumn{1}{c}{00}&\multicolumn{1}{c}{01}&\multicolumn{1}{c}{02}&\multicolumn{1}{c}{10}&
\multicolumn{1}{c}{11}&\multicolumn{1}{c}{12}&\multicolumn{1}{c}{20}&\multicolumn{1}{c}{21}&\multicolumn{1}{c}{22}&\multicolumn{1}{c}{30}&\multicolumn{1}{c}{31}&\multicolumn{1}{c}{32}\\
\cline{2-13}
\multicolumn{1}{r|}{00}&0&1&2&0&1&2&0&1&2&0&1&2\\
\multicolumn{1}{r|}{10}&0&1&2&0&1&2&0&1&2&0&1&2\\
\multicolumn{1}{r|}{20}&0&1&2&0&1&2&0&1&2&0&1&2\\
\multicolumn{1}{r|}{30}&0&1&2&0&1&2&0&1&2&0&1&2\\
\cline{2-13}
\multicolumn{1}{r|}{01}&2&0&1&2&0&1&2&0&1&2&0&1\\
\multicolumn{1}{r|}{11}&2&0&1&2&0&1&2&0&1&2&0&1\\
\multicolumn{1}{r|}{21}&2&0&1&2&0&1&2&0&1&2&0&1\\
\multicolumn{1}{r|}{31}&2&0&1&2&0&1&2&0&1&2&0&1\\
\cline{2-13}
\multicolumn{1}{r|}{02}&1&2&0&1&2&0&1&2&0&1&2&0\\
\multicolumn{1}{r|}{12}&1&2&0&1&2&0&1&2&0&1&2&0\\
\multicolumn{1}{r|}{22}&1&2&0&1&2&0&1&2&0&1&2&0\\
\multicolumn{1}{r|}{32}&1&2&0&1&2&0&1&2&0&1&2&0\\
\cline{2-13}
\end{tabular}
\end{tabular}
\end{center}
\caption{Orthogonal quasi-Sudoku Latin squares under projections $\Pi_4$ and $\Pi_3$.}\label{ex:proj}
\end{figure}

\begin{figure}[h]
\begin{center}
\begin{tabular}{c}
$\Pi_4(A_1\otimes A_2),\Pi_3(B_1\otimes B_2)$\\
\begin{tabular}{c|ccc|ccc|ccc|ccc|}
\multicolumn{1}{c}{}&\multicolumn{1}{c}{00}&\multicolumn{1}{c}{01}&\multicolumn{1}{c}{02}&\multicolumn{1}{c}{10}&
\multicolumn{1}{c}{11}&\multicolumn{1}{c}{12}&\multicolumn{1}{c}{20}&\multicolumn{1}{c}{21}&\multicolumn{1}{c}{22}&\multicolumn{1}{c}{30}&\multicolumn{1}{c}{31}&\multicolumn{1}{c}{32}\\
\cline{2-13}
\multicolumn{1}{r|}{00}&0,0&1,1&2,2&3,0&0,1&1,2&2,0&3,1&0,2&1,0&2,1&3,2\\
\multicolumn{1}{r|}{10}&3,0&0,1&1,2&0,0&1,1&2,2&1,0&2,1&3,2&2,0&3,1&0,2\\
\multicolumn{1}{r|}{20}&2,0&3,1&0,2&1,0&2,1&3,2&0,0&1,1&2,2&3,0&0,1&1,2\\
\multicolumn{1}{r|}{30}&1,0&2,1&3,2&2,0&3,1&0,2&3,0&0,1&1,2&0,0&1,1&2,2\\
\cline{2-13}
\multicolumn{1}{r|}{01}&1,2&2,0&0,1&0,2&1,0&3,1&3,2&0,0&2,1&2,2&3,0&1,1\\
\multicolumn{1}{r|}{11}&0,2&1,0&3,1&1,2&2,0&0,1&2,2&3,0&1,1&3,2&0,0&2,1\\
\multicolumn{1}{r|}{21}&3,2&0,0&2,1&2,2&3,0&1,1&1,2&2,0&0,1&0,2&1,0&3,1\\
\multicolumn{1}{r|}{31}&2,2&3,0&1,1&3,2&0,0&2,1&0,2&1,0&3,1&1,2&2,0&0,1\\
\cline{2-13}
\multicolumn{1}{r|}{02}&2,1&0,2&1,0&1,1&3,2&0,0&0,1&2,2&3,0&3,1&1,2&2,0\\
\multicolumn{1}{r|}{12}&1,1&3,2&0,0&2,1&0,2&1,0&3,1&1,2&2,0&0,1&2,2&3,0\\
\multicolumn{1}{r|}{22}&0,1&2,2&3,0&3,1&1,2&2,0&2,1&0,2&1,0&1,1&3,2&0,0\\
\multicolumn{1}{r|}{32}&3,1&1,2&2,0&0,1&2,2&3,0&1,1&3,2&0,0&2,1&0,2&1,0\\
\cline{2-13}
\end{tabular}
\end{tabular}
\end{center}
\caption{The projected squares verifying that $A_1\otimes A_2$ and
$B_1\otimes B_2$ are doubly orthogonal quasi-Sudoku Latin squares
of order $12.$}\label{ex:superimposition} 
\end{figure}

If the projections $\Pi_4$ and $\Pi_3$, given above, are replaced by the projections
\begin{eqnarray*}
\Pi_4((A_1(p,q),A_2(s,t)))&=& A_1(p,q),\mbox{ and }\\
\Pi_3((B_1(p,q),B_2(s,t)))&=&B_2(s,t),
\end{eqnarray*}
we obtain the projected   square given in Figure \ref{ex:2superimposition}. Note that here each of the twelve subsquares can be obtained by reordering the rows and/or the columns of the first subsquare. This is not the case for the  projected square given in Figure \ref{ex:superimposition}. To see this consider the first subsquare and any of the subsquares on rows $01, 11, 21$ and $31$ of  Figure \ref{ex:superimposition}, setwise the rows of these subsquares do not equal any row in the first subsquare.

\begin{figure}[h]
\begin{center}
\begin{tabular}{c}
$\Pi_4(A_1\otimes A_2),\Pi_3(B_1\otimes B_2)$\\
\begin{tabular}{c|ccc|ccc|ccc|ccc|}
\multicolumn{1}{c}{}&\multicolumn{1}{c}{00}&\multicolumn{1}{c}{01}&\multicolumn{1}{c}{02}&\multicolumn{1}{c}{10}&
\multicolumn{1}{c}{11}&\multicolumn{1}{c}{12}&\multicolumn{1}{c}{20}&\multicolumn{1}{c}{21}&\multicolumn{1}{c}{22}&\multicolumn{1}{c}{30}&\multicolumn{1}{c}{31}&\multicolumn{1}{c}{32}\\
\cline{2-13}
\multicolumn{1}{r|}{00}&0,0&0,1&0,2&1,0&1,1&1,2&2,0&2,1&2,2&3,0&3,1&3,2\\
\multicolumn{1}{r|}{10}&1,0&1,1&1,2&0,0&0,1&0,2&3,0&3,1&3,2&2,0&2,1&2,2\\
\multicolumn{1}{r|}{20}&2,0&2,1&2,2&3,0&3,1&3,2&0,0&0,1&0,2&1,0&1,1&1,2\\
\multicolumn{1}{r|}{30}&3,0&3,1&3,2&2,0&2,1&2,2&1,0&1,1&1,2&0,0&0,1&0,2\\
\cline{2-13}
\multicolumn{1}{r|}{01}&0,2&0,0&0,1&1,2&1,0&1,1&2,2&2,0&2,1&3,2&3,0&3,1\\
\multicolumn{1}{r|}{11}&1,2&1,0&1,1&0,2&0,0&0,1&3,2&3,0&3,1&2,2&2,0&2,1\\
\multicolumn{1}{r|}{21}&2,2&2,0&2,1&3,2&3,0&3,1&0,2&0,0&0,1&1,2&1,0&1,1\\
\multicolumn{1}{r|}{31}&3,2&3,0&3,1&2,2&2,0&2,1&1,2&1,0&1,1&0,2&0,0&0,1\\
\cline{2-13}
\multicolumn{1}{r|}{02}&0,1&0,2&0,0&1,1&1,2&1,0&2,1&2,2&2,0&3,1&3,2&3,0\\
\multicolumn{1}{r|}{12}&1,1&1,2&1,0&0,1&0,2&0,0&3,1&3,2&3,0&2,1&2,2&2,0\\
\multicolumn{1}{r|}{22}&2,1&2,2&2,0&3,1&3,2&3,0&0,1&0,2&0,0&1,1&1,2&1,0\\
\multicolumn{1}{r|}{32}&3,1&3,2&3,0&2,1&2,2&2,0&1,1&1,2&1,0&0,1&0,2&0,0\\
\cline{2-13}
\end{tabular}
\end{tabular}
\end{center}
\caption{The projected squares verifying that $A_1\otimes A_2$ and
$B_1\otimes B_2$ are doubly orthogonal quasi-Sudoku Latin squares
of order $12.$}\label{ex:2superimposition}
\end{figure}

\section{Quasi-Sliced Orthogonal Arrays}\label{sec:oa}

A {\em (symmetric) orthogonal array}, denoted $OA(N,k,s,t)$,  is an $N\times
k$ array with entries chosen from the set $[s]$ of levels such that for
every $N\times t$ submatrix the $s^t$ level combinations of $\overbrace{[s]\times\dots \times [s]}^t$ each occur a constant number of times.
We define a {\em sliced symmetric orthogonal
array} to be an $OA(N,k,s,t)$,
${\cal O}$, which satisfies the following property:
\begin{itemize}
\item[-] there exists a projection $\Pi$ from $[s]$ to $[s_0]$, $s_0<s$, and a partition of the rows of ${\cal O}$ into $\nu$ subarrays, ${\cal O}_i$,
such that when the $s$ levels of
$[s]$ are collapsed according to the projection $\Pi$ each ${\cal
O}_i$ forms a symmetric $OA(N_0,k,s_0,t)$.
\end{itemize}
We say $({\cal O}_1 , {\cal O}_2,\dots,
{\cal O}_\nu )$ is a  sliced symmetric orthogonal
array.

For a $k$ factor design, let $S=\{s_1,\dots,s_k\}$ denote the list of numbers of factor levels.
  An {\em asymmetric orthogonal array}, denoted $OA(N,k,S,t)$,
  is an $N\times k$ array, where the $j$th column contains entries
   of $[s_j]$, $s_j\in S$,  and for each $t$-subset  $T$ of columns,
    the  $N\times t$ submatrix defined by these columns  contains all  $\prod_{i\in T}s_i$ tuples of $\prod_{i\in T}[s_i]$ (the Cartesian product of  $[s_i]$, $i\in T$) a constant, $\lambda_T$, number of times.
We define a {\em quasi-sliced asymmetric orthogonal array} to be an
$OA(N,k,S,t)$ array, ${\cal O}$, which satisfies the following property:
\begin{itemize}
\item[-] there exist projections
$\Pi_i:[s_i]\rightarrow [s_i']$ for $i=1,\dots, k$ and a partition of the rows ${\cal O}$
into $\nu$ subarrays, ${\cal O}_i$,  such that when the sets of levels $[s_1],\dots,[s_k]$ are  respectively collapsed onto  $[s_1'],\dots,[s_k']$, $s_i'<s_i$, each ${\cal O}_i$ is an asymmetric
$OA(N_0,k,S',t)$, $S'=\{s_1',s_2',\dots,s_k'\}$.
\end{itemize}
More precisely
$({\cal O}_1 ,\dots, {\cal O}_\nu )$ is said to be a
{\em quasi-sliced asymmetric orthogonal array}. In the above definition we note that if each  projections is not  one-to-one then it is immediate that $s_i'<s_i$, for all $i$.

\begin{prop}\label{sliced}
If $A_1$ and $B_1$ are pairwise orthogonal Latin squares of order $m$ and
 $A_2$ and $B_2$ are pairwise orthogonal Latin squares of order $n$. Then there exists a
{\em quasi-sliced asymmetric orthogonal array}
$OA(m^2n^2,4,\{mn\},2)$.
\end{prop}

\begin{proof}
``Unstack'' the pair of  orthogonal quasi-Sudoku Latin squares  $A_{1}\otimes A_{2}$ and $B_{1}\otimes B_{2}$
to obtain an orthogonal array $OA(m^2n^2,4, \{mn\},2)$, ${\cal O}$, 
where row $(r,c)$ of ${\cal O}$ takes the form
\begin{eqnarray*}
\left[\begin{array}{cccc}
r,&c,&(A_{1}\otimes A_{2})(r,c),&(B_{1}\otimes B_{2})(r,c)
\end{array}\right].
\end{eqnarray*}

To verify that this is a quasi-sliced asymmetric
orthogonal array  we first provide details of the partitioning of the rows and then  the projections.

 Recall that the rows of the reordered squares $A_1\otimes A_2$ and $B_1\otimes B_2$ are labelled as $(p,s)$, 
  $0\leq p\leq m-1$ and $0\leq s\leq n-1$ and the columns are labelled as $(q,t)$, $0\leq q\leq m-1$ and $0\leq t\leq n-1$.
Now for fixed $q$ let ${\cal O}_{q}$ be the $(mn^2)\times 4$ subarray with rows indexed by $((p,s),(q,t))$, $0\leq p\leq m-1$ and $0\leq s,t\leq n-1$, where a row takes the form
\begin{eqnarray*}
\left[\begin{array}{cccc}
(p,s),&(q,t),&(A_1\otimes A_2)((p,s),(q,t)),(B_1\otimes B_2)((p,s),(q,t))
\end{array}\right].
\end{eqnarray*}
Focusing on the rows corresponding to the subarray ${\cal O}_{q}$,  before we apply the projections $\Pi_1, \Pi_2,\Pi_3,\Pi_4$, we see that for columns 1, 3 and 4 the  set of levels  is $\{(p,s)\mid 0\leq p\leq m-1, 0\leq s\leq n-1\}$, for column 2  the set of levels  is $\{(q,t)\mid  0\leq t\leq n-1\}$.
Next apply the projections
\begin{eqnarray*}
\Pi_1&:&(p,s)\rightarrow s\\
\Pi_2&:&(q,t)\rightarrow t\\
\Pi_3&=&\Pi_m\\
\Pi_4&=&\Pi_n
\end{eqnarray*}
where $\Pi_m$ and $\Pi_n$ are defined in the proof of Proposition \ref{double sudoku}. Note that each of these projections is not one-to-one.

Finally we need to prove that ${\cal O}_{q}$ is a asymmetric orthogonal array with parameters $OA(mn^2,4,S',2)$ where $S'=\{n,n,m,n\}$.

Note columns 3 and 4 of  ${\cal O}_{q}$, are the concatenation of the $n$ sets of pairs \begin{eqnarray*}
\{(\Pi_m(A_1(p,q),A_2(s,t)),\Pi_n(B_1(p,q),B_2(s,t)))\mid
0\leq p\leq m-1,0\leq t\leq n-1\};\end{eqnarray*} one set for each $0\leq s\leq n-1$.
The set
$\{(\Pi_m(A_1(p,q),A_2(s,t)),\Pi_n((B_1(p,q),B_2(s,t)))\mid
0\leq p\leq m-1,0\leq t\leq n-1\} =[m]\times [n]$. Hence columns 3 and 4 are the concatenation of $n$ copies of $[m]\times [n]$.

For fixed $s$
the set of pairs $\{(s,\Pi_m(A_1(p,q),A_2(s,t)))\mid 0\leq p\leq m-1,0\leq t\leq n-1\}$ gives $n$ copies of $\{s\}\times [m]$ and $\{(s,(\Pi_n(B_1(p,q),B_2(s,t)))\mid 0\leq p\leq m-1,0\leq t\leq n\}$ gives $m$ copies of $\{s\}\times [n]$. So, respectively, as $s$ takes the values $0,\dots,n-1$ we obtain  $n$ copies of $[n]\times [m]$ and $m$ copies of $[n]\times [n]$.

 For fixed $s$ the set $\{(t,\Pi_m(A_1(p,q),A_2(s,t)))\mid 0\leq p\leq m-1,0\leq t\leq n-1\}$ gives one copy of $[n]\times [m]$. So as $s$ takes the values $0,\dots,n-1$ we obtain $n$ copies of $[n]\times [m]$. Finally as $s$ takes the values $0,\dots,n-1$ $\{(t,\Pi_n(B_1(p,q),B_2(s,t)))\mid 0\leq p\leq m-1,0\leq t\leq n-1\}$ gives $m$ copies of $[n]\times [n]$. 
 
Thus each ${\cal O}_{q}$, for $0\leq q\leq m-1$ is an asymmetric orthogonal
array with the required parameters. In all cases the
 projections are not one-to-one and the results is a quasi-sliced asymmetric orthogonal array as required.
\end{proof}

Once again, for $m=n$, the extension of these pairwise properties to $K>2$ orthogonal direct product designs $OA(m^2n^2,2+K,\{mn\},2)$ is immediate, if the component designs are available.

By way of example we project the first three columns of the array given in Figure \ref{ex:superimposition}
to obtain the quasi orthogonal array ${\cal O}_0$ ($q=0$) with parameters
$OA(36,4,S',2)$, where $S'=\{3,3,4,3\}$.

\begin{figure}[h]
\begin{center}
${\cal O}_0$\\
\begin{tabular}{c}
0\ 0\ 0\ 0\ 0\ 0\ 0\ 0\ 0\ 0\ 0\ 0\ 1\ 1\ 1\ 1\ 1\ 1\ 1\ 1\ 1\ 1\ 1\ 1\ 2\ 2\ 2\ 2\ 2\ 2\ 2\ 2\ 2\ 2\ 2\ 2\\
0\ 1\ 2\ 0\ 1\ 2\ 0\ 1\ 2\ 0\ 1\ 2\ 0\ 1\ 2\ 0\ 1\ 2\ 0\ 1\ 2\ 0\ 1\ 2\ 0\ 1\ 2\ 0\ 1\ 2\ 0\ 1\ 2\ 0\ 1\ 2\\
0\ 1\ 2\ 3\ 0\ 1\ 2\ 3\ 0\ 1\ 2\ 3\ 1\ 2\ 0\ 0\ 1\ 3\ 3\ 0\ 2\ 2\ 3\ 1\ 2\ 0\ 1\ 1\ 3\ 0\ 0\ 2\ 3\ 3\ 1\ 2\\
0\ 1\ 2\ 0\ 1\ 2\ 0\ 1\ 2\ 0\ 1\ 2\ 2\ 0\ 1\ 2\ 0\ 1\ 2\ 0\ 1\ 2\ 0\ 1\ 1\ 2\ 0\ 1\ 2\ 0\ 1\ 2\ 0\ 1\ 2\ 0
\end{tabular}
\end{center}
\caption{The projected sliced array ${\cal O}_0$.}\label{ex:quasi-slice}
\end{figure}

\section{Construction of Quasi-Sudoku-Based Sliced Space-Filling Designs}\label{construction2}

Sliced space-filling designs can be constucted from the doubly orthogonal quasi-Sudoku Latin squares or the quasi-sliced asymmetic orthogonal array using the techniques in \cite{qianWu}, \cite{tang}, and \cite{XuHaalandQian}. In brief, the column-wise procedure is as follows. First, randomly relabel each column's elements $1,\ldots,mn$ subject to the constraint that the elements mapped to the same symbol by $\Pi$ form a consecutive subset of $1,\ldots,mn$. Next, replace the elements in each column with symbol $k$ with a random permutation of $(k-1)mn+1,\ldots,(k-1)mn+mn$. Lastly, each element of the space-filling is generated as $(x_{ij}-u_{ij})/(mn)^2$, where $x_{ij}$ denotes the design elements after the first two steps and $u_{ij}$ denotes a random ${\rm Uniform}(0,1)$ deviate. As an example, Figure \ref{spacefilling} shows a sliced space-filling design constructed based on the quasi-sliced asymmetic orthogonal array whose first sliced is shown in Figure \ref{ex:quasi-slice}. 

As shown in \cite{qianWu} and \cite{XuHaalandQian}, the complete design achieves uniformity in one and two-dimensional projections, while the slices are guaranteed to achieve uniformity in two-dimensional projections. On the other hand, the columns based on the doubly orthogonal quasi-Sudoku Latin squares can be divided into a larger number of slices, as defined by the quasi-Sudoku sub-squares, each of which has guaranteed uniformity under one and two-dimensional projections, as shown in \cite{XuHaalandQian}.

\begin{figure}[!h]
\centering
\includegraphics[height=6in,width=6in]{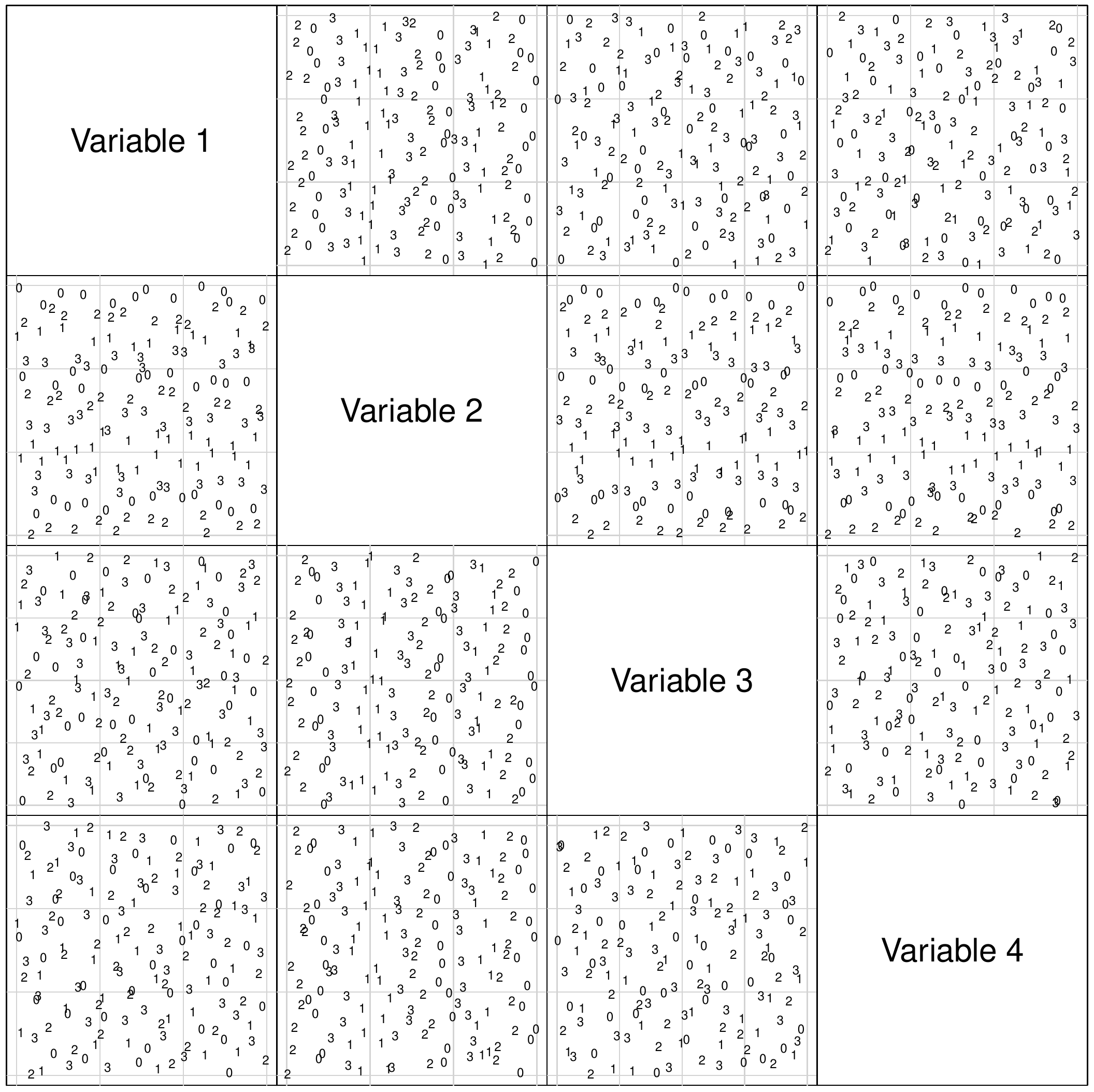}
\caption{Sliced space-filling design based on a quasi-sliced asymmetic orthogonal array. Slices indicated by point labels.}\label{spacefilling}
\end{figure}

%
%
%


\end{document}